\documentclass[oneside,english]{amsart}
\usepackage{mathptmx}
\usepackage[T1]{fontenc}
\usepackage[utf8]{inputenc}
\usepackage{fancyhdr}
\usepackage{babel}
\usepackage{url}
\usepackage{amsthm}
\usepackage{amssymb}
\usepackage[unicode=true,pdfusetitle,
 bookmarks=false,
 breaklinks=false,pdfborder={0 0 1},backref=false,colorlinks=false]
 {hyperref}
\usepackage[capitalise]{cleveref}
\usepackage{tagging}
\usepackage{xcolor}
\usepackage{enumerate}
\numberwithin{equation}{section}

\theoremstyle{plain}
\newtheorem{thm}{Theorem}
\newtheorem{prop}[thm]{Proposition}
\newtheorem{lem}[thm]{Lemma}

\newtheorem{cor}[thm]{Corollary}
\newtheorem{claim}[thm]{Claim}
\theoremstyle{definition}
\newtheorem{defn}[thm]{Definition}

\theoremstyle{remark}
\newtheorem{rem}[thm]{Remark}
\newtheorem*{acknowledgement*}{Acknowledgments}

\usepackage{mathrsfs}
\usepackage{amssymb}
\usepackage{dsfont}
\usepackage{verbatim}
\usepackage{url}
\usepackage{mathtools}
\usepackage{etoolbox}
\usepackage{leftidx}

\def\Z{\mathbb {Z}}
\def\Q{\mathbb {Q}}
\def\R{\mathbb {R}}
\def\C{\mathbb {C}}

\newcommand\Qp{\Q_p}

\def\isom{\simeq}

\def\bs{\backslash}
\def\eqdef{\overset{\text{def}}{=}}

\DeclareMathOperator{\Ad}{Ad}

\DeclareMathOperator{\diag}{diag}

\DeclareMathOperator{\Lie}{Lie}

\DeclareMathOperator{\Res}{Res}

\DeclareMathOperator{\Span}{Span}

\DeclareMathOperator{\Stab}{Stab}
\DeclareMathOperator{\supp}{supp}

\DeclareMathOperator{\vol}{vol}

\newcommand\abs[1]{\left| {#1} \right|}
\newcommand\norm[1]{\left\Vert {#1} \right\Vert}

\newcommand{\xdashrightarrow}[2][]{\ext@arrow 0359\rightarrowfill@@{#1}{#2}}

\DeclareRobustCommand
  \rddots{\mathinner{\mkern1mu\raise\p@
    \vbox{\kern7\p@\hbox{.}}\mkern2mu
    \raise4\p@\hbox{.}\mkern2mu\raise7\p@\hbox{.}\mkern1mu}}

\newcommand{\Cc}{C_{\mathrm{c}}}

\newcommand{\Cic}{\Cc^{\infty}}

\newcommand\SL{\mathrm{SL}}

\newcommand\Sp{\mathrm{Sp}}

\newcommand\gO{\mathrm{O}}
\newcommand\SO{\mathrm{SO}}

\newcommand\SU{\mathrm{SU}}

\newcommand\so{\mathfrak{so}}

\newcommand\lieg{\mathfrak{g}}
\newcommand\lieh{\mathfrak{h}}

\newcommand\liem{\mathfrak{m}}
\newcommand\lien{\mathfrak{n}}

\newcommand\lies{\mathfrak{s}}
\newcommand\liet{\mathfrak{t}}

\newcommand\liez{\mathfrak{z}}

\newcommand\liesr{\lies_{\R}}
\newcommand\liesrs{\liesr^{*}}
\newcommand\lietr{\liet_{\R}}
\newcommand\lietrs{\lietr^{*}}
\newcommand\liezr{\liez_{\R}}

\newcommand\BB{\mathbb{B}}
\newcommand\CC{\mathbb{C}}

\newcommand\GG{\mathbb{G}}
\newcommand\HH{\mathbb{H}}

\newcommand\bLL{\mathbb{L}}

\newcommand\RR{\mathbb{R}}
\newcommand\bSS{\mathbb{S}}
\newcommand\TT{\mathbb{T}}

\newcommand\ZZ{\mathbb{Z}}

\DeclareMathAlphabet{\mathcal}{OMS}{cmsy}{m}{n}

\newcommand\calF{\mathcal{F}}

\newcommand\calH{\mathcal{H}}

\newcommand\calN{\mathcal{N}}

\newcommand\calP{\mathcal{P}}

\newcommand\calR{\mathcal{R}}
\newcommand\calS{\mathcal{S}}
\newcommand\calT{\mathcal{T}}

\newcommand\fraka{\mathfrak{a}}

\newcommand\frakg{\mathfrak{g}}
\newcommand\frakh{\mathfrak{h}}

\newcommand\frakk{\mathfrak{k}}
\newcommand\frakl{\mathfrak{l}}
\newcommand\frakm{\mathfrak{m}}
\newcommand\frakn{\mathfrak{n}}

\newcommand\diff{\mathop{}\!\mathrm{d}}

\newcommand\dcross{\diff^{\kern-2pt\raisebox{1pt}{$\times$}}\kern-8pt}

\newcommand\Adele{\mathbb{A}}

\newcommand\AF{\Adele_F}

\newcommand\Af{\Adele_\mathrm{f}}

\newcommand\GQ{\GG(\Q)}
\newcommand\GF{\GG(F)}

\newcommand\GA{\GG(\Adele)}
\newcommand\GAF{\GG(\AF)}
\newcommand\GAf{\GG(\Af)}

\newcommand\Kf{K_\mathrm{f}}

\newcommand\gf{g_\mathrm{f}}

\newcommand\Qbar{\bar{\Q}}

\newcommand\Fbar{\bar{F}}

\def\cf{cf.\ }

\def\resp{resp.\ }

\newcommand\blfootnote{\xdef\@thefnmark{}\@footnotetext}

\DeclareFontFamily{U}{wncy}{}
\DeclareFontShape{U}{wncy}{m}{n}{<->wncyr10}{}
\DeclareSymbolFont{mcy}{U}{wncy}{m}{n}
\DeclareMathSymbol{\Sha}{\mathord}{mcy}{"58} 

\newcommand\mb{\bar\mu}

\DeclareMathOperator{\denom}{denom}

\begin{document}

\title{Homogeneity of arithmetic quantum limits for hyperbolic $4$-manifolds}

\author{Zvi Shem-Tov}
\address{Einstein Institute of Mathematics, The Hebrew University of Jerusalem,  Israel}
\email{\url{zvi.shem-tov@mail.huji.ac.il}}

\author{Lior Silberman}
\address{Department of Mathematics, University of British Columbia, Vancouver\ \ BC\ ~V6T 1Z2, Canada}
\email{\url{lior@math.ubc.ca}}

\begin{abstract}
We work toward the arithmetic quantum unique ergodicity (AQUE) conjecture for
sequences of Hecke--Maass forms on hyperbolic $4$-manifolds. We show that
limits of such forms can only scar on totally geodesic $3$-submanifolds,
and in fact that all ergodic components of the microlocal lift other than
the uniform measure arise from the uniform measures on these submanifolds.
\end{abstract}


\maketitle
\tableofcontents{}

%\fancypagestyle{plain}{}

\section{Introduction}
\subsection{Background}
Let $Y$ be a compact Riemannian manifold, and let $\phi_j\in L^2(Y)$ be
normalized eigenfunctions of the positive Laplace--Beltrami operator $-\Delta$
with eigenvalues $\lambda_j\to\infty$.  Interpreting the
$\phi_j$ as the stationary quantum states of a free particle moving
on $Y$ naturally leads one to be interested in classifying the subsequential
weak-$*$ limits $\mb$ of the measures $\mb_j$ defined by
$$
d\mb_j=\abs{\phi_j(x)}^2d\vol,
$$
where $d\vol$ is the Riemannian volume on $Y$.

The mathematical study of such \emph{semiclassical limits} was initiated
by Shnirelman \cite{Schnirelman:Avg_QUE}.  As an initial step, 
the measures $\mb_j$ can be lifted to the so-called \emph{Wigner distributions}
(also known as \emph{microlocal lifts}) $\mu_j$ on the unit contangent bundle $T^1 Y$
so that any subsequential limit of the lifts is a measure $\mu$ there,
lifting $\mb$ and invariant by the geodesic flow.
Shnirelman then establishes\footnote{The result was announced in
\cite{Schnirelman:Avg_QUE} with the proofs appearing in
\cite{Shnirelman:EigenfunctionsStatistics};
see also \cite{Zelditch:SL2_Lift_QE, CdV:Avg_QUE}} 
that if the geodesic flow is ergodic then every orthonormal basis of $\phi_j$ as above admits a subsequence of density $1$ such that the corresponding sequence of measures $\mb_j$ becomes equidistributed, a phenomenon known as \emph{Quantum Ergodicity}.

After investigating the case of Hecke--Maass forms on the modular surface
Rudnick and Sarnak \cite{RudnickSarnak:Conj_QUE} conjectured a stronger
property they called \emph{Quantum Unique Ergodicity} (or QUE) holds for 
$Y$ of of negative sectional curvature: that 
\emph{every} sequence $\{\mu_j\}$ as above converges weak-$*$ to the
volume measure on $T^1Y$.

While no manifold of dimension at least $2$ is known to satisfy this
conjecture, in the seminal work \cite{Lindenstrauss:SL2_QUE} Lindenstrauss
proved that it holds in the case of compact hyperbolic surfaces
$Y=\Gamma\bs \HH^2$, where $\Gamma\subset\SL_2(\R)$ is a
congruence lattice, and where the $\Delta$-eigenfunctions $\phi_j$ are,
in addition, also eigenfunctions of the \emph{Hecke operators} on $Y$ (this special case is known as \emph{Arithmetic Quantum Unique Ergodicity} or AQUE). See \cref{hecke-operators} for a discussion on Hecke operators.
For the proof Lindenstruass first used a version \cite{Lindenstrauss:HH_QUE}
of the microlocal lift described above which is compatible with the Hecke
operators. The result is a probability measure $\mu$ on the homogeneous space
$X=\Gamma\bs \SL_2(\R)$ invariant under the action of the diagonal subgroup
$A<\SL_2(\R)$ and satisfying a positive entropy condition established
in \cite{LindenstraussBourgain:SL2_Ent} (this relies on the compatibility
of the lift with the Hecke structure).  While general $A$-invariant measures can be wild,
Lindenstrauss proved that under the entropy condition and an additional
recurrence condition the measure $\mu$ must be $\SL_2(\R)$-invariant. For noncompact hyperbolic surfaces
such as the modular surface
Lindenstrauss proves that the limits $\mu$ are \emph{proportional} to the
uniform measure; the proof of full equidistribution was completed by
Soundararajan~\cite{Sound:EscapeOfMass}.

The key ingredient in \cite{Lindenstrauss:SL2_QUE} was a measure classification result for measures invariant under a 1-parameter flow with an additional recurrence property. This is closely related to the study of measures invariant under higher rank diagonalizable actions, in particular the earlier works by Katok-Spatzier \cite{KatokSpatzier:InvariantMeasures} and Einsideler--Katok\cite{EinsiedlerKatok:SLn_Rigid}, with subsequent work including \cite{EKL:SLn_Rigid,EinsiedlerLindenstrauss:GeneralMaxSplit},
as well as a result for higher dimensional rank-$1$ spaces
\cite{EinsiedlerLindenstrauss:LowEntropyMethod} (a result that is important for the purpose of this paper).
Inspired by these measure rigidity results, Silberman--Venkatesh \cite{SilbermanVenkatesh:SQUE_Lift,SilbermanVenkatesh:AQUE_Ent} formulated
the AQUE problem for general locally symmetric spaces $\Gamma\bs G/K$ and
used higher-rank measure classification results to prove AQUE-type
theorems in some higher rank setting
(see also \cite{BrooksLindenstrauss:OneHecke,ShemTov:OnePlaceHighRank}
where the arithmetic assumption is significantly weakened). 

In our recent work \cite{ShemTovSilberman:HH_AQUE_preprint} we returned to
the rank-one setup that was considered initially by Rudnick and Sarnak, 
and proved the AQUE conjecture for hyperbolic $3$-manifolds (more generally,
for the locally symmetric spaces associated to inner forms of the group
$\SL_2$ over algebraic number fields).
In the present paper we continue our study of AQUE in hyperbolic manifolds,
focusing on dimension $4$.
\subsection{Statement of results}\label{statement}
Let $Q$ denote the quadratic form 
\begin{equation}\label{theform}
Q=2x_0x_1+x_2^2+x_3^2+x_4^2,
\end{equation}
and let $\GG=\SO(Q)$ be the corresponding $\Z$-group. 
Let $G=\GG(\R)$ denote the group of real points of $\GG$ and $\Gamma=\GG(\Z)$
its group of $\Z$-points. Then $\Gamma$ is a congruence lattice in $G$, and we let $X$ denote the homogeneous space 
\begin{equation}\label{eq1}
X=\Gamma\bs G,
\end{equation} 
and $dx$ its probability Haar measure. Let $\calH$ be the Hecke algebra
of this space (see \cref{hecke-operators} for the definition). For any ring $R$ we may realize $\GG(R)$ as the group of matrices $g\in\SL_5(R)$
satisfying 
\begin{equation}\label{realization1}
g\begin{pmatrix}J&0\\0&I_3\end{pmatrix}g^T=\begin{pmatrix}J&0\\0&I_3\end{pmatrix},
\end{equation}
where $J=\begin{pmatrix}0&1\\1&0\end{pmatrix}$ and $I_3$ is the $3\times 3$
identity matrix. Under this realization we let
$A = \{\diag(\lambda,\lambda^{-1},1,1,1) \mid \lambda \in \R_{>0}\}$
and let $M\isom \{\pm I_2\}\times\SO(3)$, the compact factor of the
centralizer of $A$.

Finally let $K = \SO(5)\cap G \isom \SO(4)$, a maximal compact subgroup of
$G$ compatible with our choice of $A$, and let $Y = X/K$ be the
resulting locally symmetric space; it is covered by the hyperbolic $4$-space
$S=G/K$.

\begin{thm}\label{mainthm1}
Let $\{\phi_j\}_{j=1}^\infty\subset L^2(Y)$ be normalized joint eigenfunctions
of $\calH$ and of the Laplace--Beltrami operator $\Delta_S$ with Laplace
eigenvalues $\lambda_j \to\infty$.
Let $\mb$ be a subsequential weak-$*$ limit of the measures $d\mb_j$
with density $\abs{\phi_j}^2$ relative to the Riemannian volume on $Y$.
Then $\mb$ is proportional to a countable linear combination of measures,
each of which is either the Riemannian probability measure on $Y$ or the
Riemannian probability measure on a totally geodesic hyperbolic submanifold
of codimension $1$.
\end{thm}

Let $\{\phi_j\}_{j=1}^\infty$ be a sequence as in \cref{mainthm1}.
By the equivariant microlocal lift of \cite{SilbermanVenkatesh:SQUE_Lift},
after passing to a subsequence, there exist Hecke eigenfunctions
$\tilde{\phi}_j\in L^2(X)$ such that the measures $\mu_j$ with density
$\abs{\phi}^2$ relative to the invariant measure of $X$ converge weak-$*$ to
an $MA$-invariant measure $\mu$ lifting $\mb$.
\cref{mainthm1} then follows from the following result.
\begin{thm}\label{mainthm2}
Let $\{{\phi}_j\}_{j=1}^\infty\subset L^2(X)$ and for each $j$ let
$\mu_j$ be the probability measure on $X$ with density
$\abs{{\phi}_j}^2$ relative to $dx$.  Suppose:
\begin{enumerate}
\item The ${\phi}_j$ are normalized eigenfunctions of $\calH$.
\item The sequence of probability measures $\mu_j$ converges in the weak-$*$
topology to an $A$-invariant measure $\mu$ on $X$.
\end{enumerate}
Then $\mu$ is proportional to a convex combination of homogeneous measures.
Each of them is either the uniform measure $dx$ or supported on an orbit of
a subgroup of $G$ isomorphic to $\SO(1,3)$.  

Each non-uniform component of
$\mu$ is the $\tilde{H}$-invariant probability measure on a subset
$\Gamma ym\tilde{H}$ where $y\in \GG(\R\cap\Qbar)$, $m\in M$, and $\tilde{H}$ is a subgroup of 
$G$ isomorphic to $\SO(1,3)$ such that 
$m\tilde{H}m^{-1}$ is defined over $\R\cap\bar{\Q}$. 
\end{thm}

Abusing the terminology somewhat we will refer to measures $\mu$ as in 
\cref{mainthm2} as \emph{limits of Hecke eigenfunctions}.

\begin{rem} $X=\Gamma\bs\SO(1,4)$ is the (oriented) frame bundle of
$Y=\Gamma\bs\HH^4$ with $X/M = \Gamma\bs G/M$ the unit (co)tangent
bundle.  Thus an $MA$-invariant lift to $X$ is equivalent to a lift
to $T^1Y$ invariant by the geodesic flow.  The natural probability measure on
the unit cotangent bundle of a totally geodesic submanifold thus corresponds
to an $M$-orbit of measures as in the conclusion of \cref{mainthm2}.

There are only countably many such orbits since there are countably many 
totally geodesic submanifolds, equivalently countably many subgroups of $G$
defined over $\R\cap\Qbar$.
\end{rem} 

\begin{rem}\label{rem-general}
We believe that it should be possible to proceed along the lines of
\cite{ShemTovSilberman:HH_AQUE_preprint} and to replace $\Q$ with a totally
real field $F$ and $Q$ with a quadratic form over $F$ which has rank $1$
in at least one infinite place.  One then takes $G=\GG(F_\infty)$,
a product of groups isomorphic to $\SO(5),\SO(1,4),\Sp_4(\R)$, $K<G$
a maximal compact subgroup (thus containing all compact factors),
and a congruence lattice $\Gamma = \GG(F)\cap\Kf$.  The $\phi_j$ are then
automorphic forms which are joint eigenfunctions of the Hecke algebra
and of the Laplace--Beltrami operator at a place where the group has rank $1$,
with the eigenvalues tending to infinity.

The conclusion would be that the ergodic components of any limit measure
which are not Haar measure would be the uniform measures the totally
geodesic submanifolds correponding to the isometry groups of $4$-dimensional
quadratic subspaces of $Q$ which are isotropic at the fixed real place.

In particular when $Q$ is $F_v$-anisotropic in all but one infinite place and
has $F_v$-rank $1$ at that place there is only one noncompact factor in $G$
and $S=G/K$ would still be $\HH^4$.  In that case $\Gamma$ would
be cocompact ($X,Y$ would be compact) and the limit measure $\mu$ would
necessarily be a probability measure.
\end{rem} 

\subsection{Discussion}
\subsubsection{Proof ideas}
There is a significant qualitative difference between the AQUE problem in 4-manifolds and in 3-manifolds, a case that we handled in our previous paper \cite{ShemTovSilberman:HH_AQUE_preprint}. In order to apply the measure rigidity results from \cite{EinsiedlerLindenstrauss:GeneralLowEntropy} one needs to verify a positive entropy condition for quantum limits and a  ``non exceptional returns'' condition. Note that the positive entropy condition is needed for a.e.\ ergodic component of the limiting measure, as opposed to the total ergodic theoretic entropy for which there are general results by Annantharaman and Annantharaman--Nonnenmacher \cite{Anantharaman:QUE_Ent,AnantharamanNonnenmacher:HalfEntropyAnosov}; cf. also the paper \cite{AnantharamanSilberman:HaarComponent} by Annantharaman and the second named author.

To see the difference between dimensions $4$ and $3$ we note that e.g. to show positive entropy one needs to bound the measure of a small measure of a neighbourhood of a compact piece of the centralizer of the one-parameter Cartan subgroup of $\SO(1,d)$. In the case of $\SO(1,3)$, this group is a torus whereas in $\SO(1,4)$ it contains the simple group $M=\SO(3)$ that in terms of its behaviour at the finite places is much bigger. Whether AQUE holds for $d$-manifolds with $d$-large seems to be a completely open question, though this case is certainly covered by the original QUE conjecture of Rudnick and Sarnak.

%This paper is a continuation of our previous work
% where we proved AQUE for hyperbolic
%$3$-manifolds.  There the main novelty was a new method for proving
%non-concentration of the limit measure on certain submanifolds.  In the
%present paper we substantially improve this method, allowing us to apply
%the measure rigidity machinery in further cases.

\medskip

To illustrate our ideas we discuss the proof of the following result, which is used 
to establish the other condition, the "no exceptional returns" condition (which is Condition~\ref{cond3} below).
As we will show in \cref{few-sec}, the key step in verifying the condition is establishing the folowing claim:

\begin{claim} With the notation above, let $L=\Gamma gHM\subset X$, where $g\in G\cap\GG(\bar{\Q})$, $M$ as above and $H$ a normalizer of an anisotropic $\Q$-torus in $M$.
If $\mu$ is a limit of Hecke eigenfunctions then $\mu(L)=0$.
\end{claim}

Since $L$ can be covered by countably many open bounded subsets, we may
instead consider $L\subset gHM$ which are open and bounded.
Next, it in fact suffices to show that
$\mu(L_\delta)\xrightarrow[\delta\to 0]{} 0$ where $L_\delta\subset X$
is the $\delta$-neighborhood of $L$.  Further, $\mu$ being the weak-$*$ limit
of the $\mu_j$ this amounts to showing that $\mu_j(L_\delta)\to 0$
\emph{uniformly in} $j$. Concretely this reads: For every 
$\epsilon>0$ there exists $\delta_0$ such that for all $\delta<\delta_0$ and for all $j$,
\begin{equation}\label{toshow}
\int_{L_\delta}\abs{\phi_j}^2\le \epsilon. 
\end{equation}
A standard method for obtaining such estimates is "amplification": One
constructs a Hecke operator $\tau$ (the ``amplifier'') so that
$\tau\star\phi_j = \lambda_j(\tau) \phi_j$ with $\lambda_j$ large, and
on the other hand the \emph{translates} $s.L_\delta$ with $s$ in the support
of $\tau$ are ``smeared'' across $X$.  This reduces bounding $\mu_j(L_\delta)$
to bounding sums of the form
\begin{equation}\label{eq:sum-of-measures}
\sum_{s\in S} \mu_j\left(s.L_\delta\right)\,.
\end{equation}

In the simplest cases one tries to show these translates are disjoint, in
which case the sum is at most $\mu(X)=1$, but in general one has to contend
with the intersection pattern of the translates. Replacing $L$ by a small piece of it, we may view it as a subset of $G$, and then intersections
$s.L_\delta\cap s'.L_\delta\neq\emptyset$ occur when
$s^{-1}s' \in L_\delta L_\delta^{-1}$.
In the key situation where
$L$ is contained in (a coset of) a subgroup of $G$, say $H$, that  means that $s^{-1}s'$ must lie 
in a neighbourhood of (a conjugate of) $H$ and it is sometimes possible to
choose $\tau$ avoiding this possibility for many pairs $(s,s')$ in the
support.  Such counts played a central role in the works
\cite{LindenstraussBourgain:SL2_Ent,SilbermanVenkatesh:AQUE_Ent,Marshall:KsmallLinftyBds_preprint},
among others.
However, for our $L\subset gHM$ even $LL^{-1}$ by itself is an \emph{open}
subset of $G$, so intersections of translates are a generic feature 
and there are too many returns to estimate the sum purely by counting them.

In our previous work 
\cite{ShemTovSilberman:HH_AQUE_preprint} we found a way forward. We classified intersections
$s.L_\delta \cap L_\delta$ according to whether $s.L \cap L$ is open in $gHM$
("parallel translate") or whether it is a lower-dimensional subset
("transverse intersection").  Roughly speaking the $s$ causing parallel
translates lie in a rational subgroup and their contribution to
\eqref{eq:sum-of-measures} \emph{can} be estimated by counting them.
On the other hand transverse intersections have small $\mu$-mass by induction
on the dimension of $L$.
In greater detail, note that a transverse intersection $s.L\cap L$ is a
submanifold of $X$, but not in general of the form $g'AB$ for subgroups $A,B$.  
Thus carrying an inductive argument through requires extending the class of
subsets $L$ so that it is closed under transverse intersection, while
ensuring that it is still possible to construct Hecke operators with the
requisite properties to bound the contribution from parallel ones.

A new difficulty arises in $\HH^{4}$ in the construction of the Hecke
operators: the rational subgroups that control parallel intersections
are relatively larger in the ambient group $G$, so it is more difficult to
construct Hecke operators with sufficiently
few returns to such a subgroup.  To study this let $H$ be one of these
subgroups and $\tau= \Gamma a\Gamma$ a Hecke operator with $a\in \GG(\ZZ[1/p])$
(for some prime number $p$).  The number of parallel returns when amplifying
with $\tau$ can be estimated by the size of the intersection 
\begin{equation}\label{eq-size}
(K_p a K_p)\cap (H_pK_p)/K_p,
\end{equation}
where $K_p=\GG(\Z_p)$ and $H_p = H(\Qp)$.  If $H$ is a torus one can show
that, for a fixed $a$, the size of \eqref{eq-size} is a \emph{constant}
independent of $p$, which was ultimately used in the proofs of the positive
entropy condition in \cite{LindenstraussBourgain:SL2_Ent,SilbermanVenkatesh:AQUE_Ent}.
The groups $H$ that arise in our context, however, typically make the size of
\eqref{eq-size} \emph{polynomial} in $p$, and this loss must then be balanced
against the gain from the eigenvalue of $\tau$ acting on $\phi_j$, making
the selection of $\tau$ a delicate problem of harmonic analysis on $G_p/K_p$.
This problem was considered by Marshall \cite{Marshall:KsmallLinftyBds_preprint}
in great generality. 
Marshall gives a condition for $H$ being \emph{small} in $G$
so that for each $\phi_j$ there exist Hecke operators with eigenvalues large
enough relative to the number of returns. He shows that for many reductive
groups the maximal compact subgroup $K<G$ is small in this sense, thus
establishing non-trivial upper bounds for the \emph{sup-norm} of Hecke--Maass
forms on congruence quotients of the symmetric spaces $G/K$ by
amplification techniques.

We now combine Marshall's techniques and our methods.
Specifically we prove in that each of the stabilizers $H$ arising
in the inductive argument described above is contained in the set of real
points of a small subgroup of $G$, giving a sufficient bound on the contribution
of parallel returns.
 
\subsubsection{Strengthening the result}
Ideally we would like to establish arithmetic quantum unique ergodicity,
that is that the uniform measure is the only weak-$*$ limit.
There are two issues that prevent us from doing so. 

First, since every rational quadratic form in five variables over is isotropic,
the results stated above alway concern limits on noncompact quotients
$X$ where the total mass of $\mu$ could be less than one.  We will address
this problem of \emph{escape of mass} in future work. 
 
Second, the conclusion in \cref{mainthm1} was limited by the fact
that totally geodesic $3$-submanifolds in a hyperbolic $4$-manifolds are
"too large": the corresponding subgroups isomorphic to
$\SO(1,3)$ are not small in $\SO(1,4)$.  It seems that a new idea
would be required to rule out concentration on those submanifolds.
Further, this difficulty (unlike escape of mass) cannot be avoided by
restricting the class of lattices: \emph{every} arithmetic hyperbolic
$4$-manifold contains embedded hyperbolic $3$-manifolds.
%Avoid citation unless asked to include
%\cite[\S6.4]{Morris:IntroArithmeticGroups}

\subsection{Plan of the paper}\label{plan}
We set out some background and notations in \cref{sec:notation}.
In particular we identify the space $X=\Gamma\bs G$ with the adelic quotient
$\GQ\bs \GA/\Kf = \GQ\bs G\times \GAf/\Kf$.
We then proceed to verify hypotheses of the rank-$1$ measure classification
theorem of Einsiedler and Lindenstrauss \cite{EinsiedlerLindenstrauss:LowEntropyMethod}, and conclude by ruling out most
possible ergodic components to obtain our main theorems.  For the convenience
of the reader we restate the result here using the adelic point of view
(in the original it is written for $S$-arithmetic groups).

\begin{thm}[{\cite[Thm.\ 1.5]{EinsiedlerLindenstrauss:LowEntropyMethod}}]
\label{thm:EL-LowEntropy}
Let $\mu$ be an $A$-invariant probability measure on $X$. Suppose:
 \begin{enumerate}[(I)]
 \item\label{cond1}\emph{(Recurrence).}
    $\mu$ is $\GAf$-recurrent. That is, for every
    set $U\subset X$ of positive measure and a.e.\ $x\in U$ the set 
    $\{\gf\in \GAf\mid x\gf\in U\}$ is unbounded (has non-compact closure).
 \item\label{cond2}\emph{(Positive Entropy).}
    Almost every $A$-ergodic component of $\mu$ has positive ergodic-theoretic
    entropy with respect to some non-trivial $a\in A$.
 \item\label{cond3}\emph{(Few exceptional returns).}
    For $\mu$-a.e. $x\in X$ the group
      $$
      \{h\in M\times\GAf\mid xh=x\}
      $$
    is finite.
  \end{enumerate}
Then $\mu$ is a convex combination of homogenous measures.  Each such measure
is supported on a closed orbit of a closed subgroup subgroup $H<G$ containing
both the diagonal group $A$ and a finite-index subgroup $L$ of a semisimple
algebraic $\R$-subgroup $L'<G$ of real rank~$1$.  Moreover, the group $H$ is a
finite-index subgroup of an algebraic $\R$-subgroup of $G$.
\end{thm}
 
In \cref{sec:smallness} we prove a general non concentration result for Hecke
eigenfunctions. 
We prove that any weak-$*$ limit $\mu$ as in \cref{mainthm2} satisfies the three conditions above in \cref{rec-section}, \cref{entropy-sec} and \cref{few-sec}. Then in \cref{last} we analyze the possible measures that may appear as ergodic components of $\mu$ and conclude the proof of Theorem \ref{mainthm2}.

\begin{acknowledgement*} The authors are grateful to Manfred Einsiedler and Elon Lindenstrauss for many fruitful discussions, and to Saugata Basu and
Elina Robeva for useful suggestions during an earlier phase of this paper.
LS was supported by an NSERC discovery grant.
This material is based upon work supported by the National Science Foundation under Grant No. DMS-1926686. 
ZS was partially supported by the ERC grant HomDyn no. 833423.
\end{acknowledgement*}

\section{Conventions and notation}\label{sec:notation}
\subsection{Fields}
Let $\bar{\Q}$ denote the algebraic closure of $\Q$ in $\C$. By a \emph{real number field} we mean a finite extension of $\Q$ contained in $\bar{\Q}\cap \R$. 

\subsection{Algebraic groups}
Fix a commutative ring $k$ with unit.  In this paper by a $k$-\emph{group}, or an \emph{algebraic group over} $k$, we shall mean an affine group scheme
over $k$.  If $\GG$ is a $k$-group, and $k'$ is a $k$-algebra, 
we let $\GG_{k'}$ denote the group obtained by 
base change from $k$ to $k'$, that is the $k'$-group represented by
restricting the functor $\GG$ to $k'$-algebras.
Abusing notation we sometimes refer to $\GG$ itself as a $k'$-group if this
would not cause confusions, possibly writing $\GG/k'$.
Conversely, if $\GG$ is a $k'$-group we let $\Res_{k'/k}\GG$ denote the
$k$-group obtained by Weil restriction of scalars, determined by
$\Res_{k'/k}\GG(R)=\GG(R\otimes_k k')$ for each $k$-algebra $R$.   
Finally, when $k$ is a field, two $k$-groups $\GG,\GG'$ are $k$-\emph{forms}
of each other if they become isomorphic to $\GG$ over a field extension $L/k$ --
in other words if the $L$-groups $\GG_L$ and $\GG'_{L}$ are isomorphic. 

When $\GG$ is an algebraic group over a local field $k$, e.g. $k=\R$, its group of $k$-points $G=\GG(k)$ has
a natural structure of a Lie group over $k$. In this case we shall say that a subset of $G$ is \emph{closed} if it is closed with respect to the analytic 
topology on $G$, that is the topology induced by the topology of $k$. We call a subgroup (or more generally, a subset) $H$ of $G$ \emph{algebraic} if it is the set of $k$-points of a $k$-subgroup of $\GG$. In most of the cases we consider the group $\GG$ is actually defined over a number field $F\subset k$; that is, $\GG$ is obtained from an $F$-group $\tilde{\GG}$, say, by extension of scalars. In this case we shall say that an algebraic subgroup $H$ of $G$ is \emph{defined over $F$} if $H$ is the $k$-points of an $F$-subgroup $\HH$ of $\tilde{\GG}$. 

\subsection{Hecke operators}\label{hecke-operators}
We retain the notation of the introduction. In particular, $\GG=\SO(Q)$ is the affine group scheme over $\Z$ corresponding to the quadratic form 
$Q=2x_0x_1+x_2^2+x_3^2+x_4^2$, $G=\GG(\R)$ its group of real points and $\Gamma=\GG(\Z)$ the group of $\Z$-points, which is a congruence lattice in $G$ when the latter is viewed as a real Lie group. We recall two different ways for defining Hecke operators acting on functions $X=\Gamma\bs G$. We start with the classical approach. For this, let $a$ be an element in the commensurator ${\rm{Comm_G(\Gamma)}}$ of $\Gamma$ in $G$. Let $S\in G$ be a set of representatives for $\Gamma a\Gamma$ in $X$. Then $S$ is finite and we define
\begin{equation}\label{basic-hecke}
T_af(x)=\sum_{s\in S}f(sx), 
\end{equation}
for any function $f$ on $X$ and $x\in X$. We call any operator of the form \eqref{basic-hecke} a \emph{basic} Hecke operator. 
The basic Hecke operators generate an algebra of operators over $\C$ and every element of this algebra is called a \emph{Hecke operator}. 
One can check that any Hecke operator has the form 
\begin{equation}\label{general-hecke}
\tau f(x)=\sum_{s\in {\rm{Comm_G(\Gamma)}}}\tau(s)f(sx), 
\end{equation}
where $s\mapsto \tau(s)$ is a finitely supported, $\Gamma$-invariant, function on $\Gamma\bs{\rm{Comm_G(\Gamma)}}$. Identifying this function with $\tau$ itself, we call the support of $s\mapsto \tau(s)$ the \emph{support} of $\tau$. We shall occasionally view the support of $\tau$ as a subset of $X$.   

In the sequel we shall not use the algebra of \emph{all} the Hecke operators
as defined above, but a large commutative subalgebra of it. To describe it
we adopt an adelic point of view.  Let $\Af$ denote the ring of finite adeles and $\GAf$ the group of $\Af$
points. Then the map $x\mapsto (x,1)$ induces an isomorphism  
\begin{equation}\label{adelic-appendix}
X\isom \GG(\Q)\bs G\times \GAf /\Kf, 
\end{equation}
where $\Kf$ is an open compact subgroup of $\GAf$ (Here we rely on the fact that our form $Q$ has class number $1$). 
We define Hecke operators through \eqref{adelic-appendix} as follows. 
If $\tau$ is a compactly supported, $\Kf$ bi-invariant, integrable function on $\GAf$ we
let it act on a function $f$ on $X$ by 
\begin{equation}\label{adeles2}
\tau\star f(x)=\int_{\GAf}f(xs)\tau(s)ds.
\end{equation}
There is $p_0$ such that for all $p>p_0$ the group $\Kf$ contains a
hyperspecial maximal compact subgroup $K_p$ of $G_p$, in which case
the convolution algebra $\Cic(K_p \bs G_p / K_p)$ is commutative.
Accordingly let $\mathcal{H}$ denote the algebra generated by those operators,
that is the algebra operators $\tau\star$ where $\tau\in \Cic(K_p \bs G_p / K_p)$ is
supported on $\prod'_{p>p_0} G_p$. We call $\mathcal{H}$ the \emph{Hecke algebra of $X$}. 

We will still sometimes think of such $\tau$ as $\Gamma$-invariant functions
on $\GQ$ (as we may, due to \eqref{adelic-appendix}). We can extend such functions to $G$ (or $X$)
by assigning them the value zero off $\GQ$. From this point of view $\tau$ acts as a classical Hecke operator; its support (originally a subset of $\GAf$, or $\GAf/\Kf$) can also be seen as a finite
subset of $\Gamma\bs \GQ$.

\section{Non concentration of Hecke eigenfunctions on manifolds with small stabilizers}\label{sec:smallness}

In this section we obtain a general non concentration result for Hecke
eigenfunctions, the key technical result underlying the rest of the paper.
Informally we rule out the possibility of a Hecke eigenfunction on
$X$ concentrating near a "sufficiently small" submanifold of $X$.

In \Cref{sec:defsmall} we recall Marshall's notion of a small \emph{subgroup}
and extract the core of his paper, the existence of the Hecke operators that
rule out concentration along the subgroup.  In \Cref{sec:nonconcentration}
we use these Hecke operators to rule out concentration along
\emph{submanifolds} whose stabilizers are small subgroups.

\subsection{Small subgroups and Marshall's Theorem}\label{sec:defsmall}

Initially let $\GG$ be a reductive group defined over an algebraically closed
field, and let $\HH\subset \GG$ be a reductive subgroup.
Let $\bSS \subset \TT$ be maximal tori in $\HH\subset \GG$ respectively;
write $\Phi(\GG:\TT)$ for the associated root system in
$\lietrs = X^*(\TT)\otimes_\Z \R$ with Weyl group $W_\GG$ and let
$\lietr = X_*(\TT)\otimes_\Z \R$.
% be the corresponding apartment (for a complex
%group this is also the Lie algebra of the real Cartan subgroup).
Choosing a notion of positivity let
$\rho_\GG = \frac12\sum_{\alpha\in\Phi^+(\bSS:\TT)} \alpha$ be the half sum of
the positive roots. We similarly have $\Phi(\HH:\bSS)$, $W_\HH$, $\liesr$,
$\liesrs$ and $\rho_\HH$; we choose the notions of positivity compatibly
for convenience.

For $X\in \lietr$ set
\begin{equation}\label{normdef}
\norm{X}_\GG = \max_{w\in W_\GG} \langle\rho_\GG,w.X\rangle\,;
\end{equation}
this defines a norm on $\lietr/\liezr$ where
$\liezr = X_*(Z(\GG))\otimes_\Z \R$ is the common nullspace of the roots.

\begin{defn}\label{defn:small}Let $\HH\subset \GG$ be as above.
\begin{enumerate}
\item We say $\HH$ is \emph{tempered in $\GG$} if for all $X\in\liesr$
we have
\begin{equation}\label{temp}
\norm{X}_\HH \leq \frac12 \norm{X}_\GG\,.
\end{equation}
\item We say $\HH$ is \emph{small in $\GG$} if it is tempered and, furthermore,
either
   \begin{enumerate}
   \item $\dim \bSS < \dim \TT$; or
   \item $\bSS = \TT$ and there exists $X\in\liesr=\lietr$ such that for all $w\in W_\GG$
   we have 
   \begin{equation}
   \norm{w.X}_\HH < \frac12 \norm{X}_\GG\,.
   \end{equation}
   \end{enumerate} 
\item If $\HH'$ is a non reductive subgroup of $\GG$ we say that $\HH'$ is
      \emph{tempered} (\resp \emph{small}) in $\GG$ if it is contained in a
      reductive subgroup which is tempered (\resp small) in $\GG$. 
\item If $F$ is any field and $\HH'\subset\GG$ are $F$-groups we say $\HH'$ is
      \emph{tempered} (\resp \emph{small}) in $\GG$ if the same is true for
      $\HH'_{\Fbar} \subset \GG_{\Fbar}$.
\end{enumerate} 
\end{defn}

\begin{rem} The first condition is due to Benoist--Kobayashi
\cite{BenoistKobayashi:TemperedHomogenous1} who showed that this condition
is equivalent to the representation of $\GG(\C)$ on
$L^2\left(\GG(\C)/\HH(\C)\right)$
being tempered (appropriately phrased their results also apply to real groups).
The second (and stronger) notion is due to Marshall
\cite{Marshall:KsmallLinftyBds_preprint},
who uses the term "weakly small" for it (and "quasi-small" for what we call
"tempered").
\end{rem}

Now let $\GG$ be a reductive $\Q$-group and let $X = \GQ\bs\GA/\Kf$.
Fix a left-invariant Riemannian metric on $G = \GG(\R)$ and for $L\subset G$
write $L_\delta$ for the $\delta$-neighbourhood of $L$ in $G$ with respect
to the fixed metric.

For a subset $S\subset G$ write $\pi(S)$ for its image in $X$ under the natural projection $\pi$.  For a Hecke operator $\tau\in\Cic(\Kf\bs\GAf/\Kf)$ now write
$$\norm{\tau}_{L^1(S)} =
\sum_{\substack{s\in\supp(\tau)\\ \GQ s\Kf \in \pi(S)}} \abs{\tau(s)}\,.$$

We extract the following result from \cite{Marshall:KsmallLinftyBds_preprint}\footnote{Page references are to the arXiv version of May 2014.}.

\begin{thm}\label{thm:marshall}
With the notation above, let $E\subset \R$ be a finite extension of $\Q$ and let
$\HH\subset \GG_E$ be a small reductive subgroup, $H = \HH(\R)\subset G$
its group of real points.  Then there exists a constant $h>0$ such
that for any compact subset $\Omega \subset G$ and bounded open subset
$V \subset H$ there exists $C>0$ such that:

For every $\delta>0$ and $x\in \Omega$ there exists a set $J$ of
positive definite Hecke operators $\tau\in\Cic(\Kf\bs\GAf/\Kf)$ such that the union of the supports of all $\tau\in J$ is a finite set, and such that for every Hecke-eigenfunction
$\phi\in L^2(X)$ with unit norm there exists $\tau\in J$ so that
$\tau\star \phi = \lambda \phi$ with $\lambda >\frac12$ and such that

\begin{equation}\label{eq1-entropy}
\frac{1}{\lambda}\norm{\tau}_{L^1(xV_\delta x^{-1})}\le C \delta^h. 
\end{equation} 
\end{thm}

\begin{proof}
In \cite[\S3]{Marshall:KsmallLinftyBds_preprint} it is shown that there exists
a Hecke operator $\calT$ such that $\tau=\calT\calT^*$ has
$\norm{\tau}_{L^1(xV_\delta x^{-1})} \leq N^{2-\eta}$ where
$N\approx\delta^{-h}$ for an absolute constant $h>0$
(see the last equation in page 17) and an absolute constant $\eta>0$.
The height of each element in the
support of the operator $\calT$ constructed there is bounded above in
terms of $\delta$ (independently of $\phi$),
so that $\calT$ is taken from a set $J$ as above.
The operator $\calT$ constructed there is a sum of
$\gg_\epsilon N^{1-\epsilon}$ operators $\tau_\nu$ with
$\tau_\nu\phi=\phi$ (see the sentence before equation (16) on page 18),
so that $\calT\calT^*\phi=\lambda\phi$ with $\lambda\gg_\epsilon N^{2-2\epsilon}$. Thus we have 
$$
\norm{\tau}_{L^1(xV_\delta x^{-1})}\ll N^{-\frac{1}{2}\eta}\approx\delta^{\frac{1}{2}h\eta}, 
$$ 
as needed. 
\end{proof}

For convenience we record the following weaker form of \cref{thm:marshall}. 
\begin{cor}\label{cor:weakmarshall}
Let $\HH\subset \GG_E$ and $V\subset H$ be as in \Cref{thm:marshall}. 
Then for every $\epsilon>0$ there exists a a set $J$ of
positive definite Hecke operators such that the union of the supports of the elements in $J$ is a finite set, 
and such that for every Hecke-eigenfunction
$\phi\in L^2(X)$ with unit norm there exists $\tau\in J$ so that
$\tau\star \phi = \lambda \phi$ with $\lambda >\frac12$ and such that

\begin{equation}
\frac{1}{\lambda}\norm{\tau}_{L^1(V)}\le \epsilon. 
\end{equation} 
\end{cor}

\subsection{A sufficient condition for non-concentration on subvarieties}\label{sec:nonconcentration}
Let $\mu$ be a limit of Hecke eigenfunctions (see the remark after \cref{mainthm2}).  To avoid cumbersome notation
we will write $\mu(V)$ instead of $\mu(\pi(V))$ for subsets $V\subset G$.
When $V$ is smaller than the local radius of injectivity this has an
intuitive meaning; when examining $\mu(V)$ for larger sets we will be claiming
that $\mu(V)=0$ by writing $V$ as a countable union of smaller pieces, so 
again the meaning is clear.

Our main result in this section is \cref{machine}, which gives a sufficient
condition for a real subvariety $L$ of $G$ to be $\mu$-null.

\begin{defn}\label{def:absmall}
Call a abstract subgroup $S\subset G$ \emph{small} if there exists
a number field $\Q\subset E\subset \R$ and a small $E$-subgroup $\HH\subset G$
such that $S \subset \HH(\R)$.  Call $S$ \emph{virtually small} if it has a 
finite-index small subgroup.
\end{defn}

\begin{defn} Call a real subvariety $L\subset \GQ$ \emph{small} if for each
real subvariety $L'\subset L$ the rational stabilizer
$$\Stab_{\GQ}(L') \eqdef \{ s\in \GQ \mid sL'=L'\}$$
is virtually small.
\end{defn}

\begin{thm}\label{machine} Let $\mu$ be a limit of Hecke eigenfunctions. Then for every small real subvariety $L$ of $G$, $\mu(L)=0$.
\end{thm}

Theorem \ref{machine} asserts that for each $g\in L$
there is a neighbourhood $g\in U \subset L$ such that for all $\epsilon>0$
there is $\delta_0>0$ so that if $\delta<\delta_0$ and $\phi$ is any
normalized Hecke eigenfunction then
\begin{equation}\label{uniform}
\mu_\phi(U_\delta) < \epsilon\,.
\end{equation}

In the case where $L$ is a \emph{subgroup} of $G$ one can do better and
obtain polynomial decay for $\mu_\phi(U_\delta)$ as a function of $\delta$.
This stronger statement is used to establish the positive entropy
condition \eqref{cond2}. 
\begin{thm}\label{quanmachine}
Let $\bLL\subset\GG_E$ be a small $\R\cap\Qbar$-subgroup
of $\GG$ and let $U\subset L=\bLL(\RR)$ be an open bounded neighborhood
of the identity.  Then there exists $h>0$ such that for every compact
subset $\Omega\subset G$ there exists $C>0$ such that:

For every $x\in\Omega$ and for every normalized Hecke eigenfunction
$\phi\in L^2(X)$
\begin{equation}\label{quanuniform}
\mu_\phi(xU_\delta)\le C\delta^h.
\end{equation} 
\end{thm}

\cref{quanmachine} follows quickly from \cref{thm:marshall}, and we establish
it first.  For any $\phi\in L^2(X)$ and a measurable subset
$N\subset G$, we put $\phi_{N}\eqdef \phi\cdot 1_{\pi(N)}$.

\begin{proof} [Proof of Theorem \ref{quanmachine}]
Let $\phi$ be a normalized Hecke eigenfunction and $\tau$ a positive-definite
Hecke operator with $\phi$-eigenvalue $\lambda>0$.
For any $V\subset G$ the projection of $\phi_V$ along $\phi$ is
$\left<\phi_V,\phi\right>\phi=\norm{\phi_V}_2^2\phi$,
so we have $\phi_V=\norm{\phi_V}_2^2\phi+R$, where $R$ is orthogonal to $\phi$.
It follows that
\begin{equation}\label{specbound1}
\left<\tau\star \phi_V,\phi_V\right>=\lambda\norm{\phi_V}_2^4+\left<\tau.R,R\right>\ge\lambda\norm{\phi_V}_2^4\,,
\end{equation}
since $\tau$ is positive.
On the other hand we have  
\begin{align}\label{geom1}
\left<\tau\star\phi_V,\phi_V\right>
&= \int\sum_{s\in\supp(\tau)}\tau(s)\phi_V(gs)\bar{\phi}_V(g)dg \\ \nonumber
&\le \sum_{\substack{s\in\supp(\tau) \\ s\in VV^{-1}}}\abs{\tau(s)}\norm{\phi_V}_2^2.
\end{align}
Combining \eqref{specbound1} and \eqref{geom1} we obtain 
\begin{equation}\label{verybasic}
\mu_\phi(V) \leq \frac{1}{\lambda}\norm{\tau}_{L^1(VV^{-1})}.
\end{equation}
Now let $V= xU_\delta$. Then
$$VV^{-1} = xU_\delta U_\delta^{-1} x^{-1}
\subset x\left(UU^{-1}_{O(\delta)}\right)x^{-1},$$
where the implicit constant depends on $\Omega$ (here we use its compactness).
Since $L$ is a subgroup, $UU^{-1}$ is itself a bounded open neighbourhood
in $L$ and the result follows by \cref{thm:marshall}. 
\end{proof}

Proving Theorem \ref{machine} is more delicate since when $L$ is a general
real subvariety, we can have $UU^{-1}$ open in $G$ and hence not "small".
In the argument above we bounded 
$$\int\phi_V(gs)\overline{\phi_V}(g)dg$$
in \eqref{geom1} by $\norm{\phi_V}^2$ using Cauchy--Schwartz, but in fact
this integral is supported on $s.V\cap V$, so we might hope to make a further
gain when these intersections are small.

\begin{defn}
We say that $s\in\GQ$ is \emph{transverse} to an irreducible subvariety
$L\subset G$ if $\dim(\gamma sL \cap L)<\dim L$ for all $\gamma\in\Gamma$.
In the alternative there exists $\gamma \in \Gamma$ such that $\gamma sL = L$
in which case we say that $s$ is \emph{parallel} to $L$.
\end{defn}

Continuing with the irreducible subvariety $L$ let $V\subset L$ be open and
bounded (in the analytic topology).  Then $V$ is Zariski-dense in $L$ and
hence determines it and we can talk of $s\in \GQ$ being parallel or transverse
to $V$.  
The following result follows immediately from the definitions. 
\begin{lem}\label{stablem}
Let $V$ be an open bounded subset of an irreducible subvariety $L$ of $G$.
Then the set of elements that are parallel with respect to $V$ is exactly
$\Gamma S$ where $S\subset \GQ$ is the stabilizer
\begin{equation}\label{stab}
\Stab_{\GQ} (L) = \{s\in \GQ\mid sL=L \}.
\end{equation}
\qed
\end{lem}

We now refine the geometric/spectral argument of \cref{quanmachine}
to take account of the distinction between parallel and transverse returns.
\begin{lem}\label{basic}
Let $L$ be a submanifold of $G$ and $U=U_0\subset L$ a compact subset with nonempty interior.
Assume that $U$ is contained in a fundamental domain $\calF$ for
$\Gamma\bs G$.  

Let $\tau$ be a positive definite Hecke operator, $S\subset\GQ$ a set of representatives
for $\supp(\tau)$, which we divide into parallel and transverse elements as $S=\calP\sqcup \calT$.
Then there exists a finite collection $\calN$ of submanifolds of $U_0$ such that each
$N\in\calN$ is of the form $N=U_0\cap b U_0$ ($b\in\GQ$) and has dimension strictly lower 
than $\dim L$, and such that:
For every $\delta>0$ there exists $\delta'>0$ such that for any eigenfunction $\phi$ of $\tau$
with eigenvalue $\lambda>0$, we have 
\begin{equation}
\mu_\phi(U_{\delta'})\le
 \frac{1}{\lambda}\left(\norm{\tau}_{L^1(\calP\cap U_{\delta'}U_{\delta'}^{-1})}
+\frac{\norm{\tau}_\infty}{\mu_\phi(U_{\delta'})}\sum_{N\in\calN}\mu_\phi(N_\delta)\right).
\end{equation}
\end{lem} 

\begin{proof}
As before we begin with the expression
\begin{equation}\label{inner}
\left<\tau.\phi_{U_{\delta'}},\phi_{U_{\delta'}}\right>.
\end{equation}

Writing $V=U_{\delta'}$ the lower bound of \eqref{specbound1} still applies;
we refine the geometric arguments giving the upper bound.  For convenience we may assume
$V\subset \calF$ if $\delta'$ is small enough.
Separating $S$ into the parallel and transverse elements gives
$$\left<\tau.\phi_{U_{\delta'}},\phi_{U_{\delta'}}\right> = 
 \sum_{s\in \calP} \tau(s) \int \phi_V(gs)\overline{\phi_V}(g)dg
+\sum_{s\in \calT} \tau(s) \int \phi_V(gs)\overline{\phi_V}(g)dg\,.$$

Parallel intersections are controlled by a subgroup (the stabilizer of $L$) so their contribution
is analogous to the group case above and we simply apply Cauchy--Schwarz to each summand, getting:
$$\abs{\sum_{s\in \calP} \tau(s) \int \phi_V(gs)\overline{\phi_V}(g)dg}
  \leq \norm{\tau}_{L^1(\calP\cap VV^{-1})} \norm{\phi_V}^2\,.$$

For each $s\in \calT$ we have 
$1_{\Gamma V}(sg)1_{\Gamma V}(g)\ne0$ if and only if 
$g\in V\cap s^{-1}\gamma^{-1}V$ for some $\gamma\in \Gamma$
such that $\gamma sg\in V$. First, let $\gamma\in \Gamma$ be one such element. 
By the compactness of $U_0$, for each $\delta>0$ there exists $\delta'>0$
(depending also on $\gamma$ and $s$), such that 
$$
U_{\delta'}\cap s^{-1}\gamma^{-1}U_{\delta'}\subset (U_0\cap s^{-1}\gamma^{-1} U_0)_\delta.
$$ 
Multiplying by $\gamma s$ on the left we see that dimension of $U_0\cap s^{-1}\gamma^{-1} U_0$
is equal to the dimension of $\gamma sU_0\cap U_0$, which is strictly less than $\dim L$
since $s$ is transverse to $L$. Denoting $N=U_0\cap s^{-1}\gamma^{-1} U_0$ we have,
\begin{align*}
\abs{\int_{N_{\delta}}\phi(sg)\overline{\phi(g)}dg}
  &\le \frac{1}{2}\left(\int_{N_{\delta}}\abs{\phi(g)}^2dg+\int_{N_{\delta}}\abs{\phi(sg)}^2dg\right) \\
&=\frac{1}{2}\left(\int_{N_{\delta}}\abs{\phi(g)}^2dg+\int_{(\gamma sN)_{\delta}}\abs{\phi(g)}^2dg\right) \\
&= \frac12 \left(\mu_\phi(N_\delta) + \mu_\phi(\gamma sN_\delta)\right)\,.
\end{align*}
Summing over the finitely many $\gamma\in\GQ$ such that $\gamma s U_1\cap U_1$ is nonempty
we cover the image of $sV\cap V$ in $X$.  Since, further, there are finitely many $s\in \calT$
there are finitely many relevant pairs $(s,\gamma)$ and we obtain finitely many manifolds of
the required form (collect them together as $\calN$).  In addition, by the finiteness we can
choose $\delta'$ independently of $(s,\gamma)$ depending only on $U_0$.  In summary we have obtained:
$$\sum_{s\in \calT} \tau(s) \int \phi_V(gs)\overline{\phi_V}(g)dg \leq
  \norm{\tau}_\infty \sum_{N\in\calN}  \mu_\phi(N_\delta)\,.$$

Putting everything together we get
$$\lambda\norm{\phi_V}_2^4 \leq \norm{\tau}_{L^1(\calP\cap VV^{-1})} \norm{\phi_V}^2 
+ \norm{\tau}_\infty \sum_{N\in\calN}  \mu_\phi(N_\delta)\,.$$
Dividing by $\lambda$ and by $\norm{\phi_V}_2^2 = \mu_\phi(V)$ gives the claim.
\end{proof}

\begin{proof}[Proof of Theorem \ref{machine}]
Since our assumptions on $L$ holds for any real subvariety of it as well, and in particular for each of the irreducible components of $L$, of which there are at most finitely many, we may assume that $L$ is irreducible. 
Let $U$ be a bounded open subset of $L$ contained in a fundamental domain $\calF$ for $\Gamma\bs G$.
Since $L$ is a union of countably many such $U$, it is suffices to show $\mu(U)=0$. 
Given $\epsilon>0$ we will find $\delta>$ such that $\mu_\phi(U_\delta)\leq\epsilon$ for all
Hecke eigenfunctions $\phi$, by induction on $\dim L$.  

Accordingly let $L$ and $U$ as above, and assume the statement is true for any proper subvariety
$L'$ of $L$.  Let $S\subset \Stab_{\GQ}(L)$ be a finite-index subgroup such that $S\subset H=\HH(\R)$ where $\HH\subset\GG$ is a small subgroup defined over $\Qbar\cap\R$. Since $S$ has finite index
in $\Stab(L)$, if $V$ is a small enough neighbourhood of the identity in $H$ we have $V\cap \Stab(L)\subset S$.  In particular we can take $V=U_r U_r^{-1} \cap H$ if $U$ is small enough
(which we may assume with loss of generality) and $r>0$ is small enough.

Now given $\epsilon>0$ \cref{cor:weakmarshall} provides a set $J\subset \calH$ of
positive definite Hecke operators, whose supports are jointly contained in a finite set
and such that for every Hecke eigenfunction $\phi$ there is $\tau\in J$
eigenvalue $\lambda>\frac12$, and so that 
\begin{equation}
\frac{1}{\lambda}\norm{\tau}_{L^1(\rm{Stab}(L)\cap V)}\le \epsilon.
\end{equation}
By \cref{basic} there exists a finite collection $\calN$
of bounded open subsets of real subvarieties of $L$, of dimension strictly
lower than $\dim L$, depending on the common support of the operators in $J$
but not on each $\tau\in J$ separately,
such that for every $\delta>0$ there exists $\delta'>0$ such that
\begin{align*} \mu_\phi(U_{\delta'})
& \leq \frac{1}{\lambda}\left(\norm{\tau}_{L^1(\calP\cap U_{\delta'}U_{\delta'}^{-1})}
+\frac{\norm{\tau}_\infty}{\mu_\phi(U_{\delta'})}\sum_{N\in\calN}\mu_\phi(N_\delta)\right) \\
& \leq 2\epsilon + \frac{2\norm{\tau}_\infty}{\mu_\phi(U_{\delta'})}
\sum_{N\in\calN}\mu_\phi(N_\delta).
\end{align*}

By the induction hypothesis, there exists $\delta_0>0$ such that if
$\delta<\delta_0$ we have
$\frac{2\norm{\tau}_\infty}{\mu_\phi(U_{\delta'})}\sum_{N\in\calN}\mu_\phi(N_\delta)<\epsilon$.
It follows that for every $\epsilon>0$ there exists $\delta'>0$ such that
$$
\mu_\phi(U_{\delta'})\ll\sqrt{\epsilon}, 
$$
completing the proof.
\end{proof}

\section{Recurrence}\label{rec-section}
Let $F$ be a number field, $\GG/F$ a reductive group.  Let $G=\GG(F_\infty)$.
Let $\Kf<\GAf$ be an open compact subgroup and let $X=\GF\bs\GAF/\Kf$.
For a normalized $\phi\in L^2(X)$ let $\mu_\phi$ be the probability measure
with density $\abs{\phi}^2$ with respect to the $\GA$-invariant measure on
$\GF\bs\GAF$.  

In this section we verify Condition \ref{cond1} from \cref{thm:EL-LowEntropy}
for any measure $\mu$ which is a weak-* limit of such $\mu_\phi$.
We will use the following result.
\begin{lem}[{\cite[Lem.\ A.1]{SilbermanVenkatesh:AQUE_Ent}}]
There exist $\ell,\ell'>0$ (depending on $\GG$) and $p_0>0$ (depending also
on $\Kf$) such that for all $p>p_0$ and Hecke eigenfunctions
$\phi\colon\GAf/\Kf$ there exists
a Hecke operator $\tau_p$ such that:
\begin{enumerate}
\item Its support satisfies
$ p^{\ell'} \gg  \# \supp(\tau_p) \gg  p$, where we think of $\tau_p$ as a function on $G_p/K_p$;
\item $|\tau_p| \in \{0,1\}$; 
\item $\tau_p \star \phi = \Lambda(p) \phi$ where
$\Lambda(p)\gg \# \supp(\tau)^{1/2}$
\item For any $s \in \supp(\tau_p)$, the denominator of $s$ is
$\ll_N p^{\ell}$.
\end{enumerate}
\end{lem}

Now, by the arguments of \cite[\S8]{Lindenstrauss:SL2_QUE} (see especially
the deduction of Cor.\ 8.4 from Lemma 8.3 and the deduction therefrom of
Thm.\ 8.1), it suffices to prove the following:

\begin{lem}[\cf {\cite[Lem.\ 8.3]{Lindenstrauss:SL2_QUE}}] \label{previous}
There exists $\alpha>0$
such that for all Hecke eigenfunctions $\phi\colon \GAf/\Kf$ and $N\geq 1$
we have
$$\sum_{\substack{s\in \GAf/\Kf\\ \denom(s)\leq N}} \abs{\phi(s)}^2 \gg N^\alpha \abs{\phi(e)}^2\,.$$
\end{lem}
Note that the implied constant is independent of $\phi$.

\begin{proof}
For each $p$ let $\tau_p$ be as in \cref{previous}.  Then
$\Lambda(p) \phi(e) = \sum_{s\in \supp(\tau_p)} \phi(s.e)$ so by
Cauchy--Schwarz
$$\abs{\phi(e)}^2 \leq \frac{\# \supp(\tau_p)}{\abs{\Lambda(p)}^2}
  \sum_{s\in\supp(\tau_p)} \abs{\phi(s)}^2\,.$$

Summing over all primes $p_0 < p < N$ we obtain
\begin{align*}
\frac{N}{\log N}\abs{\phi(e)}^2
 & \ll \sum_{p\leq N} \sum_{s\in\supp(\tau_p)} \abs{\phi(s)}^2 \\
 & \leq \sum_{\substack{s\in \GAf/\Kf \\ \denom(s)\ll N^{\ell}}} \abs{\phi(s)}^2\,.
\end{align*}

Rearranging and renaming $N$ we obtain for any $\alpha < \frac{1}{\ell}$

$$ N^\alpha \abs{\phi(e)}^2 \ll
   \sum_{\substack{s\in \GAf/\Kf \\ \denom(s)\ll N}} \abs{\phi(s)}^2\,.$$
\end{proof}

\begin{rem} We use the amplifier above because it exists in the literature
in a convenient form; essentially any amplifier would do here, including the
ones used later to prove the positive entropy and few exceptional returns
conditions, which are based on this one through Marshall's work.
\end{rem}

\section{Positive entropy}\label{entropy-sec}
In this section we will prove the positive entropy condition \ref{cond2}. 
In general, suppose that $\mu$ is a probability measure on a quotient $X=\Gamma\bs G$ where $G$ is a semisimple Lie group and $\Gamma$ a lattice in $G$, and suppose that $\mu$ is invariant under the action of an element $a\in G$. Let $C_G(a)\subset G$ denote the centralizer of $a$ in $G$. The following result is well-known.
\begin{lem}\label{wellknown1}
Let $\mu$ be an $a$-invariant measure on $X$.  Suppose that for any
open bounded bounded $U\subset C_G(a)$ and $x\in X$ there exist $C,h>0$
so that $\mu(xU_\delta)\le C\delta^h$.
Then almost every $a$-ergodic component of $\mu$ has positive
ergodic-theoretic entropy with respect to $a$.\qed
\end{lem}

Since $G\isom\SO(1,4)$ has real rank $1$, every regular $a\in A$
has $C_G(a)=C_G(A)=AM$. Positivity of the entropy for limits of Hecke
eigenfunctions then follows immediately from \cref{quanmachine} upon
verifying that $AM$ is small.

For later reference we also consider another subgroup of $G$.  Let
$K_1<K$ be any connected subgroup isogenous to $SO(3)\times\SO(2)$
(there are two conjugacy classes in $K$, but they are conjugate in $G$).

\begin{lem}\label{smallnesslem1}
The subgroups $AM,K_1 \subset \SO(1,4)$ are small.
\end{lem}
\begin{proof}
Let $\HH,\HH'\subset\GG$ be the $\Q$-subgroups such that $\HH(\R) = AM$,
$\HH'(\R)\isom K_1$.  Then $\HH(\C)$ and $\HH'(\C)$ are both isomorphic
to $S(\gO(2,\C)\times\gO(3,\C))$ but are not conjugate in $\GG(\C)$.

Concretely we can realize $\GG(\CC)$ as the set of $g\in\SL_5(\C)$ so that
$$
g\begin{pmatrix}J&0&0\\0&J&0\\0&0&1\end{pmatrix}g^T=\begin{pmatrix}J&0&0\\0&J&0\\0&0&1\end{pmatrix},
$$
where $J=\begin{pmatrix}0&1\\1&0\end{pmatrix}$, then $\HH(\CC)$ is the
obvious subgroup of block-diagonal matrices with blocks of size $2$ and $3$.

In this representation all three groups share the maximal torus
$$\bSS = \TT = \{ \diag(s,s^{-1},t,t^{-1}) \vert s,t\in\C^\times\}\,,$$
and we write $e_1,e_2$ for the characters of this torus given by taking the
values $s$, $t$ respectively.  The positive roots of $\GG_\C$ are then
$\{e_1 \pm e_2\} \cup \{e_1,e_2\}$ of which $\{e_2\}$ is the positive root
of $\HH_\C$ (note that $\SO(2)$ is a torus), while $e_1+e_2$ is
the positive root of $\HH'_\C$.  We then have
$\rho_\HH = \frac12 e_2$, $\rho_{\HH'} = \frac12e_1 + \frac12e_2$,
and $\rho_\GG = \frac32 e_1 + \frac12 e_2$.

The absolute Weyl group of $\GG$ is the signed permutation group
$W=S_2\ltimes C_2^2$ acting on $X^*(\TT)\isom \Z e_1\oplus\Z e_2$
by swapping the two coordinates and reversing signs in
each coordinate independently.  The absolute Weyl group of both subgroups
is $C_2$, reversing the sign of the respective root.

Now let $X=(X_1,X_2) \in \lietr$ so that 
$\langle X,\rho_\HH\rangle = \frac12 X_2$,
$\langle X,\rho_{\HH'}\rangle = \frac12 X_1+\frac12 X_2$, and
$\langle X,\rho_\GG\rangle = \frac32 X_1 + \frac12 X_2$.
Acting by $W$ we may assume $X_1\geq X_2\geq 0$ and then
$\max_{w\in W} \norm{w.X}_\HH = \frac12 X_1$, and
$\max_{w\in W} \norm{w.X}_\HH = \frac12(X_1+X_2)$, and
$\norm{X}_\GG = \frac32 X_1 + \frac12 X_2$.

For $\HH$ we immediately obtain $\norm{X}_\GG>2\norm{X}_\HH$ for all
non-zero $X$.  For $\HH'$ we see
$\frac32 X_1 + \frac12 X_2 =  X_1 + \frac12(X_1+X_2) \geq 2\cdot \frac12(X_1+X_2)$,
and that the inequality is strict if $X_1>X_2$.
\end{proof}

\section{Few exceptional returns}\label{few-sec}
In this section we prove \ref{cond3}.  We will show that the set of $x$ as
in the condition is contained in countably many submanifolds of the form
$L=\Gamma gHM$, with $H$ a certain copy of $\SO(1,2)$ in $\SO(1,4)$ --
the centralizer of a compact torus in $G$, to be described below.
We will show that such $L$, as well as any of its subvarieties, has
small stabilizers, at which point the claim will follow from \cref{machine}.

For future reference we analyze the stabilizers of analogous subvarieties
of $G=\SO(1,n)$ for general $n$.

\subsection{Block matrices conventions}\label{bmconv}
Let $I_n$ be the $n\times n$ identity matrix, $J=\begin{pmatrix}0&1\\1&0\end{pmatrix}$.
Realize $G=\SO(1,n)$ as a $\bar{\Q}\cap \R$-group as the special isometry
group of the quadratic form
$$
Q=2x_0x_1+x_2^2+\dots+x_n^2. 
$$
Concretely $G$ is then the group of $g\in M_{n+1}(\RR)$
such that
$$g\begin{pmatrix}J&0\\0&I_{n-1}\end{pmatrix}g^t
= \begin{pmatrix}J&0\\0&I_{n-1}\end{pmatrix}\,.$$
We will write elements of $G$ and its lie algebra $\lieg$ as block matrices
corresponding to the decomposition of the quadratic space as a sum of
a hyperbolic space and an $n-1$-dimensional definite space.
In particular $\lieg$ the consists of those $X\in M_{n+1}(\R)$ such that
\begin{equation}\label{4blocks}
X=\begin{pmatrix}
a&b\\c&d
\end{pmatrix}\,,
\end{equation}
where $a\in\Span_\R(\diag(1,-1))\subset M_2(\R)$,
$b\in M_{2\times{n-1}}(\R)$ is arbitrary, $c=-b^tJ$, and
$d=-d^t\in M_{n-1}(\R)$.
We let $M$ denote the copy of $\SO(n-1)$ in the lower right corner and
$\frakm$ its Lie algebra; in the coordinates above it consists of those
matrices where $a,b,c$ all vanish.
Let $N$ denote the (abelian) unipotent group whose Lie algebra $\frakn$
consists of elements of the form 
\begin{equation}\label{typicaln}
X=\begin{pmatrix}
0&b\\c&0
\end{pmatrix},
\end{equation}
where $b$ is any matrix of the form
$b=\begin{pmatrix}R\\0\end{pmatrix}\in M_{2\times n-1}(\R)$
(thus determining $c$ by to the relation above).
We will denote the matrix $X$ in \eqref{typicaln} by 
\begin{equation}\label{typicaln2}
X=X(R).
\end{equation} 
 
Let $\fraka$ denote the Cartan subalgebra 
$\fraka=\begin{pmatrix}
*&0\\0&0
\end{pmatrix},$ writing $A$ for the corresponding real Cartan subgroup of $G$.
Finally we let $\frakk$ denote the intersection of $\frakg$ with the algebra of antisymmetric matrices in $M_{n+1}(\R)$ and $K\subset G$ the corresponding maximal compact subgroup (isomorphic to $\gO(n)$). 
Then we have the Iwasawa decomposition $G=KAN$. 

To introduce coordinates on the group $K$ let $q$ be the quadratic form
$q=-x_0^2+x_1^2+x_2^2+\dots+x_n^2$ which also has signature $(1,n)$.  The
quadratic spaces $(\R^{n+1},q)$ and $(\R^{n+1},Q)$ are then isomorphic by the
map
$(\R^{n+1},q)\ni v \mapsto \begin{pmatrix}E&0\\0&I_{n-1}\end{pmatrix}v\in (\R^{n+1},Q)$, 
where $E=\frac{1}{\sqrt{2}}\begin{pmatrix}-1&1\\1&1\end{pmatrix}$. 
This induces an isomorphism $T:\SO(q)\to \SO(Q)$ given by
\begin{equation}\label{isomT}
g\mapsto \begin{pmatrix}E&0\\0&I_{n-1}\end{pmatrix}g\begin{pmatrix}E&0\\0&I_{n-1}\end{pmatrix}, 
\end{equation}
since $E=E^{-1}$.

Writing $(\R^{n+1},q)$ as the direct sum of a space of signature $(1,0)$ and
a space of signature $(0,n)$ gives an embedding of $\gO(n)$ in $\SO(q)$ arising
from the isometry group of $\R^n$.  It is not hard to check that its pullback
in $\SO(Q)$ via $T$ is exactly $K$, realizing $K$ as the group of matrices
of the form 
\begin{equation}\label{typicalk}
\begin{pmatrix}
E\begin{pmatrix}1&0\\0&x_1\end{pmatrix}E & E\begin{pmatrix}0\\ u\end{pmatrix}\\
\begin{pmatrix}0& v^t\end{pmatrix}E & m
\end{pmatrix},
\end{equation}
where $x_1\in \R$, $u,v\in M_{1\times n-1}(\R)$ and $m\in M_{n-1\times n-1}$ are such that  
$\begin{pmatrix}x_1&u\\v^t&m\end{pmatrix}\in \gO(n)$. 

We shall also use a refinement of the representation \eqref{4blocks} above
and write $b=\begin{pmatrix} b_1&b_2  \end{pmatrix}$, 
$c=\begin{pmatrix} c_1\\c_2  \end{pmatrix}$, and
$d=\begin{pmatrix} d_1&d_2\\d_3&d_4  \end{pmatrix}$,
where $b_1,c_1,d_1\in M_2(\R)$, representing $g\in G$ in the form
\begin{equation}\label{9blocks}
\begin{pmatrix}
a&b_1&b_2\\c_1&d_1&d_2\\c_2&d_3&d_4
\end{pmatrix}\,,
\end{equation} 
subjected to the relations above. 

Using this convention let $H$ denote the copy of $\SO(1,n-2)$ in $G$ whose Lie
algebra is
$$
\frakh=\begin{pmatrix}
*&0&*\\0&0&0\\**&0&*
\end{pmatrix}.
$$

\subsection{Unipotent subgroups of $HM$}
We bound the stabilizers of subvarieties of $gHM$ by studying their
actions on unipotent subgroups in $HM$.  In this section we determine
those subgroups.

\begin{lem}\label{lemcase1}
For each $g\in G$ either the linear transformation
$\calR:\frakn \to \R^{n-1}$ given by mapping 
an element $X\in \frakn$ to the last $n-1$ coordinates of the first row of $gXg^{-1}$ is invertible, or the 
linear transformation $\calS:\frakn \to \R^{n-1}$ given by mapping 
an element $X\in \frakn$ to the last $n-1$ coordinates of the second row of $gXg^{-1}$ is invertible. 
\end{lem}

\begin{proof}
Since $N$ is commutative and $A$ acts on $\lien$ by scalar multiplication,
and since $G=KAN$, it suffices to consider the case of $g\in K$.
Those elements have the form
$$
\begin{pmatrix}
H\begin{pmatrix}1&0\\0&x_1\end{pmatrix}E & E\begin{pmatrix}0\\ u\end{pmatrix}\\
\begin{pmatrix}0& v^t\end{pmatrix}E & m
\end{pmatrix},
$$
with $x_1,u,v,m$ as in \eqref{typicalk}. 
For $X=X(R)$ (see \eqref{typicaln2}), a direct computation shows that 
\begin{align}
2\calR X &=( 1+x_1)Rm^t -\left<u,R\right>v \,, \label{subr1}\\
2\calS X &=(-1+x_1)Rm^t -\left<u,R\right>v\,. \label{subr2}
\end{align}
Suppose first that $u=0$ (which forces $v=0$ as well).
Then $m$ is invertible, at least one of $x_1\pm 1$ is non-zero, and it
follows that either $\calR$ or $\calS$ is invertible as claimed.

Otherwise $u\neq 0$ which forces $x_1\ne\pm1$ too.  Taking $R=u$ we get
$\lambda_1,\lambda_2\in\R$ not both zero so that
\begin{align}
2\calR X(u) &= -( 1+x_1)x_1v-\norm{u}^2v =\lambda_1v \,, \label{subu1}\\
2\calS X(u) &= -(-1+x_1)x_1v-\norm{u}^2v =\lambda_2v \,. \label{subu2}
\end{align}
Thus either the image of $\calR$ or the image of $\calS$ contains $v$,
and then this span also contains the span of all the rows of $m^t$.
Since that has dimension $n-1$ we are done.  
\end{proof}

\begin{prop}\label{lem:stable-unipotent}
Let $U$ be a unipotent subgroup of $G$ (i.e. a conjugate of a closed
subgroup of $N$). If $U$ is contained in $HM$ then $\dim U \leq n-3$.
\end{prop}
\begin{proof}
We have $\Lie U \subset T_e HM = \lieh \oplus \liem$, so it suffices
to show that the maximal subspaces of $\lieh\oplus\liem$ which are
$\Ad(G)$-conjugate to subspaces of $\lien$ are at most $n-3$-dimensional.
Thus let $g\in G$ and $W\subset \lien$ be such that
$gWg^{-1} \subset \lieh\oplus\liem$.  Any element of $\liem$ has the form
$X=\begin{pmatrix}
*&0&*\\0&*&*\\**&*&*
\end{pmatrix}$, so the image of $W$ is at most $n-3$-dimensional.
\end{proof}

\subsection{Proof of condition \ref{cond3}}
We now return to the case of $n=4$ where $G\isom\SO(1,4)$,
$M\isom\gO(3)$, and $H\isom\SO(2,1)$.
\begin{lem}\label{zlem2}
Suppose that $\GG$ is a $\Q$-group and $\HH,\bLL$ are
$\Qbar\cap\R$-subgroups of $\GG$ with groups of real points $G,H,L$
receptively.  Let $g\in G$ and suppose that $gLg^{-1}=H$.
Then there exists $g'\in\GG(\Qbar\cap\R)$ such that 
$g\in g'N_G(L)$. 
\end{lem}
\begin{proof}
The theory of real closed fields eliminates quantifiers, so it is
model-complete. Thus the existential statement of conjugacy over $\R$
implies the existential statement (conjugacy) in the submodel $\Qbar\cap\R$.
Having $g'\in\GG(\Qbar\cap\R)$ with $g'Lg'^{-1}=H=gLg^{-1}$ means that
$g\in G'N_G(L)$.
\end{proof}

\begin{lem}\label{zlem1}
Let $g\in G$, and suppose that the intersection
$\GG(\Qbar\cap\R)\cap gMg^{-1}$ is infinite.
Then there exists $g'\in\GG(\Qbar\cap\R)$ so that $g\in g'HM$. 
\end{lem}
\begin{proof}
Let $S\subset G$ denote the Zariski closure of $\GG(\Qbar\cap\R)\cap gMg^{-1}$.

The proper closed connected subgroups of $\SO(3)$ are all tori, so the
infinite closed subgroups of $M$ are, up to conjugacy, $M$ itself, the tori,
and the normalizers of tori, and that the normalizer of any closed subgroup
must normalize the connected component.
We write $M_1$ for a particular torus, the copy of $\SO(2)$ in the upper
left corner of $M$, observing that it (and its extension by its Weyl
group) is defined over $\Q$.  Finally $N_G(M_1) = H$, $N_G(M)=AM$ and both
are contained in $HM$.

Since $g^{-1}Sg$ is conjugate in $M$ to a closed subgroup, $S$ is the
conjugate some $gm$ ($m\in M$) to one of our "standard" closed subgroups,
which are defined over $\Q$, so $gm$ is in an algebraic coset of the
normalizer and we obtain $g'\in\GG(\Qbar\cap\R)$ such that $gm\in g'HM$
hence that $g\in g'HM$.
\end{proof}

We keep the notation $M_1$ for that particular maximal torus of $M$.
Let $U$ denote the $1$-dimensional unipotent subgroup of $G$ whose Lie
algebra is generated by $X[(0,0,1)]$ (see \eqref{typicaln2}). 

\begin{lem}\label{real-stabilizer}
Let $\bSS$ be a connected real algebraic subgroup of $\GG$ and
$S = \bSS(\R)\subset G$ its group of real points. Suppose that $S\subset HM$. 
Then $S$ is conjugate in $G$ to a subgroup of one of the following:
\begin{enumerate}
\item A proper subgroup of $K$.
\item $AM$
\item $H M_1$
\end{enumerate}
\end{lem}
\begin{proof}
Let 
$$
\bSS=\bSS_l\ltimes\bSS_u
$$
be the Levi decomposition of $\bSS$ with $\bSS_u$ unipotent and $\bSS_l$ 
reductive (the "Levi factor").

Suppose first that $\bSS_u$ is trivial
and $\bSS$ is reductive.  It is an almost direct product 
$$
\bSS=\bSS_0\times \TT_0,
$$
where $\bSS_0$ is semisimple and $\TT_0$ is a torus. If $\TT_0$ isn't compact
it contains a conjugate of $A$ and then up to conjugacy $\bSS \subset AM$.
Otherwise it $\TT_0$ is compact (and hence contained in a conjugate of $K$).
Unless $\bSS_0$ is non trivial we are done, so assume it is not, hence
of dimension $3$ ($\dim HM=6$ and there are no semisimple Lie algebras
of dimensions $1,2,4,5$).  $\bSS_0$ is then either isogenous to $\SO(1,2)$
or to $\SO(3)$.  In the first case $\bSS_0\subset HM$ is conjugate to $H$,
so up to conjugacy $\TT_0$ is contained in $M_1$ and $\bSS\subset HM_1$.
In the latter case $\bSS$ is compact hence conjugate to a subgroup of $K$.

Now consider the case where $\bSS$ is not reductive. By
\cref{lem:stable-unipotent} its unipotent radical is one dimensional;
without loss of generality it is equal to $U$.
The normalizer of $U$ in $G$ is $M_1AN$ and so $S\subset M_1AN$.
If $m_1an$ is any element of $S$ where $m_1\in M_1,a\in A,n\in N$ then since
$S$ is algebraic we have that both the semisimple part $m_1a$, and the
unipotent part $n$ belong to $S$. Thus $n\in U$ and we conclude that
$S\subset M_1AU\subset M_1H$, finishing the proof of the lemma.
\end{proof}

The next step is showing that each of the groups specified in
\cref{real-stabilizer} is small, and we begin with the first possibility.

Let $\HH$ denote Hamilton's quaternions with group of units
$\HH^1 \isom \SU(2)$. Then $H^1 \times H^1$ acts on $H\isom \R^4$ by left-
and right-multiplication preserving the quaternion norm, giving a
representation $\SU(2)\times \SU(2) \to \SO(4)$ which turns out to be
a double cover, giving an isomorphism of Lie algebras
\begin{equation}\label{classic}
\so(3)\oplus\so(3)\isom\so(4).
\end{equation} 
Connected subgroups of $\SO(4)$ can then be classified via subalgebras
of $\so(3)\oplus\so(3)$, recalling (as already used above) that the only
proper subalgebras of $\so(3)$ are the one-dimensional subspaces, which
are all conjugate to each other. 
 
\begin{lem}\label{so(4)subgroups1}
Let $\liet\subset \so(3)$ be a maximal torus.  Then
Every non-commutative subalgebra of $\so(3)\oplus\so(3)$
conjugate in $\SO(4)$ to either a product $\frakl_1\oplus\frakl_2$ where
$\frakl_1,\frakl_2\in\{\so(3),\liet,0\}$, or to 
$\diag(\so(3)\oplus\so(3))$. 
\end{lem}

\begin{proof}
Suppose $\frakl$ is a subalgebra of $\so(3)\oplus \so(3)$. Let $\pi_L$ and $\pi_R$ denote the projection maps to the left and right copies of $\so(3)$ and let $\frakl_L$ and $\frakl_R$ denote the images of $\frakl$ under them. Let $I_L$ be the ideal of $\frakl_L$ of all the elements $X\in\frakl_L$ so that $(X,0)\in\frakl$ and similarly define $I_R$. Then $\frakl_L/I_L\isom\frakl_R/I_R$ by the map $\phi: X\mapsto Y$ if $(X,Y)\in \frakl$. It is easy to see that $\frakl$ is determined uniquely by $\frakl_L,\frakl_R,I_L,I_R,\phi$. If $\frakl_L=I_L$ then $\frakl_R=I_R$ and $\frakl=\frakl_L\oplus\frakl_R$. If $\frakl_L=\so(3)$ and $I_L=0$ then $\frakl_R=\so(3), I_R=0$ and $\phi$ is a conjugation by some element in $\SO(3)$ which without loss of generality we may assume is the identity element and we get the diagonal copy of $\so(3)$. Otherwise after conjugation we have $\frakl_L=\frakl_R=\frakh$ and $I_L=I_R=0$ so that $\frakl$ is the one dimensional subspace of $\liet\oplus\liet$ given by $\frakl=\{(X,\phi(X))\mid X\in\liet\}$ where $\phi$ is an isomorphism of $\liet$. 
\end{proof}

\begin{lem}\label{middle-lem}
Any of the groups appearing in the three items of \cref{real-stabilizer}
is conjugate in $G$ to a subgroup of the group of real points of a small
reductive subgroup of $\GG$ (defined over some real number field). 
\end{lem}
\begin{proof}
The groups $HM_1$ and $AM$ are clearly rational and reductive, as are the
closed subgroups of $K$.  A closed connected subgroup of $K$ which is not
a torus is either the diagonal $\so(3)$ (in terms of the \cref{so(4)subgroups1}),
in other words $M$, or is contained in the subgroup $K_1$ with Lie algebra
$\so(3)\oplus\so(2)$.  Since $HM_1$ and $K_1$ have the same complexification
(up to conjugacy) all three cases are covered by \cref{smallnesslem1}.
\end{proof}

\begin{lem}\label{prop:small-stabilizers} 
Let $S\subset\GG(\Qbar\cap\R)$ be a subgroup contained in some conjugate
of $HM$.  Then $S$ is virtually small (see \cref{def:absmall}).
\end{lem}
\begin{proof}
Let $\bSS$ be the connected component of the Zariski closure of $S$.
Then $\bSS(\R)$ is contained in a conjugate of $HM$, so by
\cref{real-stabilizer} there is a subgroup $\BB \subset \GG_E$ for
some real field $E$ so that $\BB(\R)$ contains a conjugate of $\bSS(\R)$.
Since the embedding of $\Qbar\cap\R$ in $\R$ is elementary, it now follows
that $\BB$ contains a $\GG(\Qbar\cap\R$-conjugate of $\bSS$.
By \cref{middle-lem} $\BB$ is small.  We thus have
$S\cap\bSS(R)\subset g\BB(\R)g^{-1}$ for some $g\in\GG(\Qbar\cap\R)$ 
and this makes $S$ virtually small since the Zariski closure has finitely
many connected components. 
\end{proof}

The following is equivalent to condition \ref{cond3}.

\begin{thm}\label{cond3dim4}
Let $X$ and $\mu$ be as in Theorem \ref{mainthm2}.
Then the set of points $x=\Gamma g\in X$ ($g\in G$) so that
$\GQ\cap gMg^{-1}$ is infinite is $\mu$-null. 
\end{thm}
\begin{proof}
Let $L=gHM$ with $g\in \GG(\Qbar\cap\R$).  We make the following claim:
for any real subvariety $L'\subset L$,
its \emph{algebraic} stabilizer $\{ g\in \GG(\Qbar\cap\R) \vert gL'=L'\}$
is virtually small.  It will follow that the rational stabilizers are
virtually small and hence that $L$ is small.  \cref{machine} would then
show that $\mu(L)=0$ and by \cref{zlem1} these $L$ cover the set under
consideration.

Next observe that replacing $L$ by $HM$ would merely conjugate the algebraic
stabilizers by $g$ so we may assume $g=1$.

Now let $L'\subset L=HM$ be a real subvariety with algebraic stabilizer $S$.
Fix $h\in H$ and $m\in M$ such that $hm\in L'$.  Then for any $s\in S$
we have $shm\in L'\subset L=HM$, so 
$$
s\in HMh^{-1}=hHM h^{-1}. 
$$ 
Thus $S\subset\GG(\Qbar\cap\R)\cap hHMh^{-1}$ and is virtually small
by \cref{prop:small-stabilizers}.
\end{proof}

\section{Proof of \cref{mainthm2}}\label{last}
In the previous sections we verified the hypotheses of
\cref{thm:EL-LowEntropy} as applied to a limit measure $\mu$, obtaining
the conclusion that it is a linear combination of homogeneous measures
supported on closed orbits of closed subgroups.  In this section we will
show that a single conjugacy class of subgroups $H$ furnishes the stabilizers
of those measures (aside from the $G$-invariant measures),
obtaining \cref{mainthm2}.

\begin{lem}\label{lem:ELproperties}
Let $H<G$ be an algebraic subgroup of $G$ such that
\begin{itemize}
\item $H$ contains $A$.
\item $H$ contains a finite index subgroup $L$ of a semisimple algebraic
subgroup $L'$ of $G$ of real rank $1$. 
\item There exists $x\in G$ such that $xHx^{-1}\cap \Gamma$ is a
lattice in $xHx^{-1}$. 
\end{itemize}
 Then 
\begin{enumerate}
\item If $x\in G$ is such that $\Gamma\cap xHx^{-1}$ is a lattice in
$xHx^{-1}$ then $xHx^{-1}$ is defined over $\Q$.
\item $H$ is conjugate in $G$ to a subgroup $\tilde{H}$ of $G$ defined over
$\Qbar\cap\R$ containing $A$. 
\item $xH \subset y\tilde{H}M$ for some $y\in\GG(\Qbar\cap\R)$.
\end{enumerate}
\end{lem}
\begin{proof}
Fix $x\in G$ such that $\Gamma\cap xHx^{-1}$ is a lattice in $xHx^{-1}$.
Then $xHx^{-1}$ is unimodular (it contains a lattice), and hence,
since it contains $A$, reductive (else the uniformly expanding action of
$A$ on the unipotent radical would make the modular character of
$xHx^{-1}$ non trivial).

The center of $H$ must be contained in the centralizer $AM$ of $A\subset H$.
In fact, $AM$ is the centralizer of any $am\in AM$ where $a$ is non-trivial, so 
either $H\subset AM$ or the center of $H$ is contained in $M$.  The first
case is impossible since $H$ contains a semisimple subgroup of $G$ of real
rank $1$.  Since the center of $H$ is compact it is the almost direct product
of a semisimple group $H_0$ without compact factors and a compact group.
Then $\Gamma\cap xH_0x^{-1}$ is a lattice in $xH_0x^{-1}$
so by the Borel Density Theorem $xH_0x^{-1}$ is defined over $\Q$.

A real-rank $1$ semisimple subgroup of $G=\SO(1,4)$ containing $A$
is determined by the subgroup of $N$ it contains, and is hence
one of $\SO(1,m)$ for $2\leq m\leq 4$ (up conjugacy and finite index).
Their centralizers are (up to finite index) the complementary subgroups
$\SO(4-m)$.  When $m=2$ this is a one-dimensional group and the groups are 
trivial otherwise.  It follows that (up to finite index) $H$ is either
$H_0$ or its product with its centralizer, and hence that $xHx^{-1}$ is
defined over $\Q$ as well.

Since $xHx^{-1}$ has rank $1$ over $\RR$, the same is true over
$\Qbar\cap\R$.  Accordingly let $S\subset xHx^{-1}$ be a maximal
$\Qbar\cap\R$-split torus defined over $\Qbar\cap\R$.
By the conjugacy of such tori in $\GG$ there exists
$y\in\GG(\Qbar\cap\R)$ so that $yAy^{-1}=S\subset xHx^{-1}$.
Thus $A\subset y^{-1}xHx^{-1}y$ with $\tilde{H}=y^{-1}xHx^{-1}y$
defined over $\Qbar\cap\R$ as claimed.

Now both $y^{-1}xAx^{-1}z$ and $A$ are maximal Cartan subgroups of $H$ so
there is $h\in \tilde{H}$ for which $hAh^{-1} = y^{-1}xAx^{-1}y$ and hence
$y^{-1}x \in hN_G(A) \subset \tilde{H} N_G(A)$.  Multiplying by $y$ on the
left we conclude that 
$$x\in y \tilde{H} N_G(A) = y\tilde{H} M\,.$$
\end{proof}

\begin{proof}[Proof of \cref{mainthm2}.]
Let $\nu$ be an ergodic component of a limit measure $\mu$.
It is a homogeneous measure invariant by group of finite index in a
subgroup $H$ satisfying the hypotheses of \cref{lem:ELproperties}.
Thus $\mu$ is contained in the countable union of the submanifolds of the
form $\Gamma y\tilde{H}M$, where $y$ is
algebraic and $\tilde{H}$ is a rational subgroup of $G$.
Further, the non compact factor of each $\tilde{H}$ is isogenous to
$\SO(1,m)$ for some $2\leq m \leq 4$.  In the previous section we
showed that when $m=2$ the manifold $y\tilde{H}M$ is small hence $\mu$-null,
and it follows that all components of $\mu$ other than (possibly) the
Haar component are invariant by subgroups isogenous to conjugates
of $\SO(1,3)$ and lie on the countable union $\Gamma y\tilde{H}M$
where $\tilde{H}$ runs over the algebraic subgroups in that conjugacy class.
\end{proof}
 
\bibliographystyle{plain}
\bibliography{que,ergodic_theory,aut_forms,lie_gps}
\end{document}